\documentclass{amsart}
\usepackage{pinlabel}
\usepackage{amsmath, amsthm, amssymb, amscd, mathrsfs, eucal, epsfig,color}
\usepackage{hyperref}
\usepackage{wrapfig}
\usepackage{epigraph}

\input xy
\xyoption{all}

\begin{document}

\newtheorem*{theorem_A}{Recursive Formula (Thm.~\ref{thm:recursive_formula})}
\newtheorem*{theorem_B}{Gap Theorem (Thm.~\ref{thm:gap_theorem})}
\newtheorem{theorem}{Theorem}[section]
\newtheorem{lemma}[theorem]{Lemma}
\newtheorem{corollary}[theorem]{Corollary}
\newtheorem{conjecture}[theorem]{Conjecture}
\newtheorem{proposition}[theorem]{Proposition}
\newtheorem{question}[theorem]{Question}
\newtheorem*{answer}{Answer}
\newtheorem{problem}[theorem]{Problem}
\newtheorem*{main_theorem}{Main Theorem}
\newtheorem*{claim}{Claim}
\newtheorem*{criterion}{Criterion}
\theoremstyle{definition}
\newtheorem{definition}[theorem]{Definition}
\newtheorem{construction}[theorem]{Construction}
\newtheorem{notation}[theorem]{Notation}
\newtheorem{convention}[theorem]{Convention}
\newtheorem*{warning}{Warning}
\newtheorem*{assumption}{Simplifying Assumptions}

\theoremstyle{remark}
\newtheorem{remark}[theorem]{Remark}
\newtheorem*{apology}{Apology}
\newtheorem{historical_remark}[theorem]{Historical Remark}
\newtheorem{example}[theorem]{Example}
\newtheorem{scholium}[theorem]{Scholium}
\newtheorem*{case}{Case}

\def\SS{\mathcal S}
\def\T{\mathbb T}
\def\N{\mathbb N}
\def\R{\mathbb R}
\def\Z{\mathbb Z}
\def\C{\mathbb C}
\def\D{\mathbb D}
\def\root{\text{root}}

\newcommand{\marginal}[1]{\marginpar{\tiny #1}}

\title{Combinatorics of the Tautological Lamination}
\author{Danny Calegari}
\address{Department of Mathematics \\ University of Chicago \\
Chicago, Illinois, 60637}
\email{dannyc@math.uchicago.edu}

\date{\today}

\begin{abstract}
The {\em Tautological Lamination} arises in holomorphic dynamics as a 
combinatorial model for the geometry of 1-dimensional slices of the Shift Locus.
In each degree $q$ the tautological lamination defines an iterated sequence of 
partitions of $1$ (one for each integer $n$) into numbers of the form 
$2^m q^{-n}$. Denote by $N_q(n,m)$ the number of times $2^mq^{-n}$ arises
in the $n$th partition. We prove a recursion formula for $N_q(n,0)$, and a
gap theorem: $N_q(n,n)=1$ and $N_q(n,m)=0$ for $\lfloor n/2 \rfloor < m < n$.
\end{abstract}

\maketitle

\setcounter{tocdepth}{1}
\tableofcontents

\section{Introduction}

The {\em Tautological Lamination}, introduced in \cite{Calegari_sausage}, is a
combinatorially defined object which gives a holomorphic model for certain 1 complex
dimensional slices of the {\em shift locus}, a fundamental object in the theory
of holomorphic dynamics. 
There is a shift locus $\SS_q$ for each degree $q$; it is the space of
depressed monic polynomials $z^q + a_2 z^{q-2} + a_3 z^{q-3} + \cdots + a_q$ in a 
complex variable $z$ (thought of as a subset of $\C^{q-1}$ with coordinates $a_j$)
for which every critical point is in the attracting basin of infinity.

There is a tautological lamination $\Lambda_T(C)$ 
for each degree $q$ and for each choice of 
{\em critical data} $C$ (certain holomorphic parameters which determine the slice of
$\SS_q$). For the complex dynamics reader: the tautological lamination records the
combinatorics of the 1 complex dimensional slices of
the shift locus where $q-2$ critical B\"ottcher coordinates are fixed, and one critical
point (with a smaller escape rate than any of the others) is allowed to vary.

Each tautological lamination determines a sequence of operations, called
pinching, which cut the unit circle $S^1$ up into pieces and reglue them into a collection
of smaller circles, denoted $S^1 \mod \Lambda_{T,n}(C)$. Subsequent operations refine
the previous ones, so each component of $S^1 \mod \Lambda_{T,n}(C)$ is cut up and
reglued into a union of components of $S^1 \mod \Lambda_{T,n+1}(C)$. The precise 
cut and paste operations depend on $C$, but the set of {\em lengths} of the components
of $\Lambda_{T,n}(C)$ (counted with multiplicity) 
depends only on $n$ and the degree $q$. These lengths are all
of the form $2^m q^{-n}$ for various non-negative integers $m$, and we can  
define $N_q(n,m)$ to be the number of components of $S^1 \mod \Lambda_{T,n}(C)$ of
length $2^m q^{-n}$.

The {\em short components} of $S^1 \mod \Lambda_{T,n}(C)$ are those with length 
$q^{-n}$. The number of short components is $N_q(n,0)$. 
Our first main result is an exact recursive formula for $N_q(n,0)$ (which can
be solved in closed form):

\begin{theorem_A}
$N_q(n,0)$ satisfies the recursion 
$N_q(0,0) = 1$, $N_q(1,0)=(q-2)$ and
$$N_q(2n,0) = qN_q(2n-1,0) \text{ and } N_q(2n+1,0) = qN_q(2n,0)-2N_q(n,0)$$
and has the generating function $(\beta(t)-1)/qt$ where
a closed form for $\beta(t)$ is given in Proposition~\ref{proposition:closed_form}.
\end{theorem_A}

At the other extreme, there is a unique largest component of $S^1 \mod \Lambda_{T,n}(C)$
of length $2^n q^{-n}$. Our second main result is a `gap' theorem:

\begin{theorem_B}
$N_q(n,m)=0$ for $\lfloor n/2 \rfloor< m < n$.
\end{theorem_B}

Both the recursive formula for $N_q(n,0)$ and the existence of a gap were observed
experimentally. Our main motivation in writing this paper was to give a rigorous
proof of these observations.

One of the striking things about the tautological lamination is the existence of a
rather mysterious bijection between the components of $S^1 \mod \Lambda_{T,n}(C)$ and
some seemingly unrelated objects called {\em tree polynomials}, introduced in
\S~\ref{section:tree_polynomials}. This bijection is a corollary of one of the 
main theorems of \cite{Calegari_sausage}, and the proof there is topological. We
know of no direct combinatorial proof of this bijection, and raise the question
of whether one can be found.

\section{Unbordered words}

Some words end like they begin, such as {\tt abra$\cdot$cad$\cdot$abra} and 
{\tt b$\cdot$aoba$\cdot$b}. Such words are said to be {\em bordered}.
Others (most) are {\em unbordered}.
A {\em border} is a {\em nonempty, proper} suffix of some word which 
is equal to a prefix.

If a word contains a border, then it contains one of at most half the
length (for, a border of more than half the length will itself be bordered 
and now we can apply induction).

If $W$ is a word, let's denote its length by $|W|$.
If $W$ is an unbordered word of even length, we can write it as $W=W_1W_2$
where $|W_1|=|W_2|$, and then for every letter {\tt c} the word $W_1\text{\tt c}W_2$ is
also unbordered. If $W$ is an unbordered word of odd length, we can write it
as $W=W_1W_2$ where $|W_1|+1=|W_2|$, and then for every letter {\tt c} the word
$W_1\text{\tt c}W_2$ is unbordered {\em except} when $W_1\text{\tt c}=W_2$. Thus:
if $a_n$ denotes the number of unbordered words of length $n$ in a $q$-letter
alphabet, then $a_0=1$ (there is one empty word) and 
$$a_{2n+1}=qa_{2n} \text{ and } a_{2n}=q a_{2n-1} - a_n$$

Let's define a generating function $\alpha(t):=\sum_{n=0}^\infty a_n t^n$.
Then the recurrence becomes
the functional equation
$$\alpha(t) = \frac{2- \alpha(t^2)} {1-qt}$$
Iteratively substituting $t^2$ for $t$ and being careful about convergence, 
one obtains the following formula: 
$$\alpha(t) = 1 + q\sum_{j=0}^\infty (-1)^jt^{2^j}\prod_{i=0}^j \frac {1} {(1- q\cdot t^{2^i})}$$

These facts are not new. Unbordered words have been studied by many authors. 
They are 
also called {\em bifix-free}, and {\em primary} (neither of these terms seem very descriptive to us). 
As far as we know they were first considered by Silberger \cite{Silberger}; 
see also e.g.\/ \cite{Lothaire}, p.~153.

A minor variation on this idea is as follows. Let's take for our $q$-letter alphabet
the elements of $\Z/q\Z$. If $W$ is a word in the alphabet, let $W'$ denote the
result of adding 1 to the first letter (digit). Say a word is {\em 1-unbordered} if
no suffix $S$ is equal to a prefix $P$ or to $P'$ (and say it is 1-bordered otherwise).
Then reasoning as above gives:

\begin{proposition}[Recursion]\label{proposition:recursion}
Let $b_n$ denote the number of 1-unbordered words of length $n$ in a $q$-letter alphabet.
Then $b_0=1$ and
$$b_{2n+1}=qb_{2n} \text{ and } b_{2n}=qb_{2n-1} - 2b_n$$
\end{proposition}
 
Define the generating function $\beta(t):=\sum_{n-0}^\infty b_n t^n$. Then
$$\beta(t) = \frac{3 - 2\beta(t^2)} {1-qt}$$

The following `closed form' for $\beta(t)$ (and the argument below) 
was kindly provided by Frank Calegari:

\begin{proposition}[Closed form solution]\label{proposition:closed_form}
The generating function $\beta(t)$ converges for small $|t|$, and can be meromorphically
continued throughout the unit disk with a simple pole at every $2^k$th root of $1/q$.

Define a sequence of integers $h(n)$ by
$$h(0):=1 \text{ and } h(n):=(-q)^{s(n)}(1-(-2)^{k(n)}) \text{ for }n>0$$
where $2^{k(n)}$ is the biggest power of $2$ dividing $n$, and $s(n)$ is the sum of
the binary digits of $n$. Then throughout the unit disk,
$$\beta(t)= \left( \sum_{n=0}^\infty h(n) t^n \right) \prod_{j=0}^\infty \frac{1} {(1-qt^{2^j})}$$
\end{proposition}
\begin{proof}
From the growth rate of the coefficients it's clear that $\beta(t)$ has a pole
at $q^{-1}$ and converges uniformly throughout the open disk of radius $q^{-1}$. 
It follows that $\beta(t^2)$ converges uniformly throughout the open disk 
of radius $q^{-1/2}$. Using the identity $(1-qt)\beta(t)= 3-2\beta(t^2)$ 
and induction, the first claim is proved.

Let's define $H(t):=\sum_{n=0}^\infty h(n)t^n$ and
$B(t):=H(t) \prod_{j=0}^\infty (1-qt^{2^j})^{-1}$.
Then the proposition will follow if we can show that $B(t)$ satisfies
$B(t)(1-qt)=3-2B(t^2)$.
First observe that $h(n)=0$ if $n$ is odd; and furthermore,
$$\frac {2h(n) + h(2n)}{3} = \frac {(-q)^{s(n)}}{3}\left(3 - 2(-2)^{k(n)} + 2(-2)^{k(n)}\right) = (-q)^{s(n)}$$

The required identity is equivalent to
$$(1-qt)B(t)\prod_{k=1}^\infty (1-qt^{2^k}) = (3-2B(t^2))\prod_{k=1}^\infty (1-qt^{2^k})$$
or
$$B(t)\prod_{k=0}^\infty (1-qt^{2^k}) = (3-2B(t^2))\prod_{k=0}^\infty (1-q(t^2)^{2^k})$$
or
$$H(t)+2H(t^2) = 3\prod_{k=0}^\infty (1-q(t^2)^{2^k})$$
Since $h(n)=0$ for $n$ odd this is equivalent to
$$\sum_{n=0}^\infty h(2n)t^{2n} + \sum_{n=0}^\infty 2h(n)t^{2n} = 3\prod_{k=0}^\infty (1-q(t^2)^{2^k})$$
Replacing $t^2$ by $t$ and using $h(2n)+2h(n)=3(-q)^{s(n)}$ this is equivalent to
$$\sum_{n=0}^\infty (-q)^{s(n)}t^n = \prod_{k=0}^\infty (1-qt^{2^k})$$
which is clear.
\end{proof}

The definition of 1-unbordered words would seem utterly unmotivated --- except that it 
just so happens that they arise naturally in an entirely different context
which is the subject of the rest of the paper.

\section{Tautological Laminations}\label{section:lamination}

\subsection{Laminations}

A {\em leaf} is an unordered pair of distinct points in a circle $S$. 
Two leaves in $S$ are {\em linked} if they are disjoint (as subsets
of $S$) and each separates the other in $S$. 
A {\em lamination of $S$} is a set of leaves in $S$, no two of which are linked.
A finite lamination is one with finitely many leaves. 

If $\Lambda$ is a finite lamination of $S$ we may {\em pinch} $S$ along $\Lambda$.
This means that we quotient each leaf to a point, so that $S$ collapses to a
`tree' of smaller circles (sometimes called a {\em cactus}), 
and then split this tree apart into its constituent circles.
We denote the result by $S \mod \Lambda$. 
See Figure~\ref{pinch}.

If there is a Riemannian metric on $S$
then we get a Riemannian metric on $S \mod \Lambda$, so it makes sense to talk about the
{\em length} of the components of $S \mod \Lambda$, and observe that the sum of these
lengths is equal to the length of $S$.

\begin{figure}[htpb]
\centering
\includegraphics[scale=0.25]{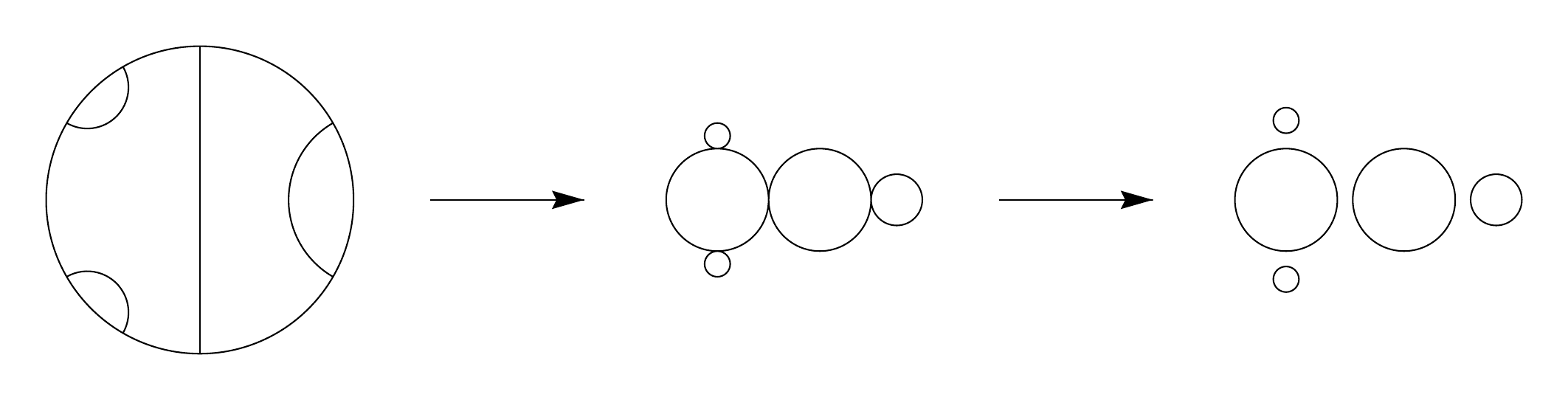}
\caption{Pinching a circle along a finite lamination to obtain a collection 
of smaller circles.}
\label{pinch}
\end{figure}

Now suppose $\Lambda$ is the increasing union of $\Lambda_n$ (for $n=1$ to $\infty$)
where each $\Lambda_n$ is finite. The {\em depth $n$ leaves} are those in 
$\Lambda_n - \Lambda_{n-1}$ and for each $n$ we can form $S \mod \Lambda_n$
for each $n$ and obtain in this way a sequence of increasingly refined partitions of $|S|$.

\subsection{Tautological Elaminations and Complex Dynamics}

We are interested in some naturally occurring laminations called 
{\em Tautological Laminations}. These objects were introduced in \cite{Calegari_sausage}
to study the geometry and topology of the {\em shift locus} --- a certain 
parameter space that arises naturally in holomorphic dynamics. For example, in degree
2, the shift locus is the complement (in $\C$) of the Mandelbrot set.

The Tautological Laminations in \cite{Calegari_sausage} have some extra structure ---
they are actually `extended laminations' or {\em elaminations}. If we identify the
circle $S^1$ with the boundary of the closed unit disk $\bar{\D} \subset \C$, 
leaves in a lamination $\Lambda$ corresponds to (infinite, unoriented) geodesics
in $\D$ thought of as the hyperbolic plane in the Poincar\'e disk model. The
unlinking property of leaves in a lamination corresponds to the condition that
the geodesics in $\D$ they span are disjoint (except at their ideal `endpoints').
In an elamination these geodesics extend beyond $S^1$ to a pair of radial
segments in $\C - \bar{\D}$. An elamination determines a lamination of $S^1$ 
(or equivalently, a geodesic lamination of $\D$) by
forgetting these `extended' segments.

As mentioned in the introduction, the tautological lamination records the
combinatorics of the 1 complex dimensional slices of
the shift locus where $q-2$ critical B\"ottcher coordinates are fixed, and one critical
point (with a smaller escape rate than any of the others) is allowed to vary.
The extra structure of the tautological elamination records not only the combinatorics,
but the holomorphic structure on these slices.

A finite elamination may be {\em pinched}, giving rise to a planar Riemann surface
which may be (partially) compactified by a finite collection of circles, which
are precisely the result of pinching the associated lamination of $S^1$.
Figure~\ref{elam_and_pinched} gives an example, approximating an infinite 
(tautological) elamination.

\begin{figure}[htpb]
\centering
\includegraphics[scale=0.4]{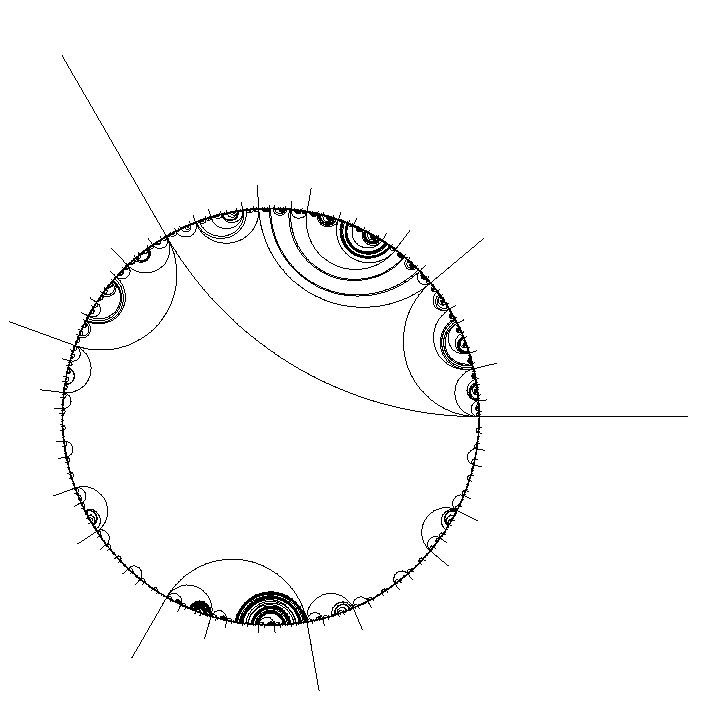} \quad \quad \quad
\includegraphics[scale=0.3]{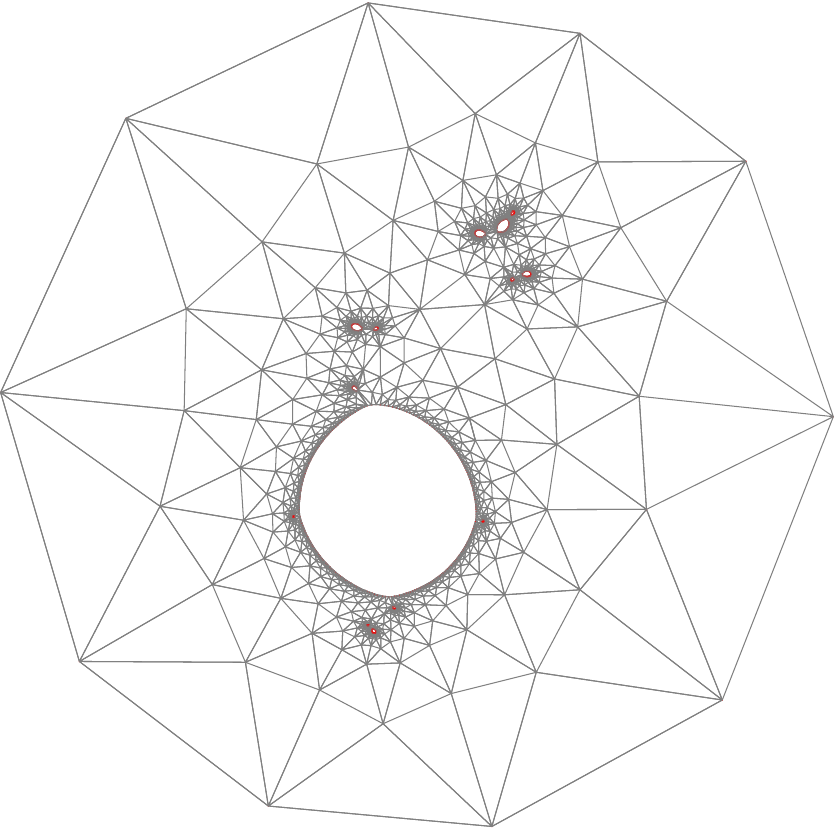}
\caption{A finite elamination approximating the Degree 3 Tautological Elamination for $z=2$, 
and the result of pinching}
\label{elam_and_pinched}
\end{figure}

To orient the reader and to motivate the remainder of this paper, let us now
describe the relationship between the tautological elamination
and the holomorphic geometry of the shift locus, in the special case of degree 3. 
A depressed cubic polynomial $f(z):=z^3+pz+q$ is in the shift locus $\SS_3$ 
if the critical points $c_1,c_2$ (not necessarily distinct)
are in the basin of attraction of infinity.
These critical points have canonical {\em B\"ottcher coordinates} $C_1$, $C_2$, 
whose absolute
value is well-defined and strictly greater than $1$, 
and whose arguments are multi-valued, where
different values differ by multiples of $2\pi/3$. For $z \in \C - \bar{\D}$
let us define the {\em B\"ottcher slice} $B(z)$ of $\SS_3$ 
to be the 1-complex dimensional subset
where $C_1=\lbrace z, e^{2\pi i/3}z\rbrace$ and
$|C_1|>|C_2|$. The open dense subset of $\SS_3$ for which the critical points
are distinct and their B\"ottcher coordinates have distinct absolute values
is foliated by such B\"ottcher slices, and in fact the B\"ottcher slices form
the fibers of a (topological) fiber bundle over $\C - \bar{\D}$.
Associated to each $z$ is a tautological elamination $\Lambda_T(z)$, and
the B\"ottcher slice $B(z)$ is obtained from $\Lambda_T(z)$ by pinching.

Figure~\ref{elam_and_pinched} depicts the tautological elamination 
$\Lambda_T(z)$ for $z=2$ and
the Riemann surface obtained from $\Lambda_T(z)$ by pinching.
The laminations of $S^1$ associated to $\Lambda_T(z)$ are the main objects of
interest throughout this section; they depend only on the argument of $z$.
We shall give them a precise definition in \S~\ref{taut_lam_subsection}.

\begin{figure}[htpb]
\centering
\includegraphics[scale=0.25]{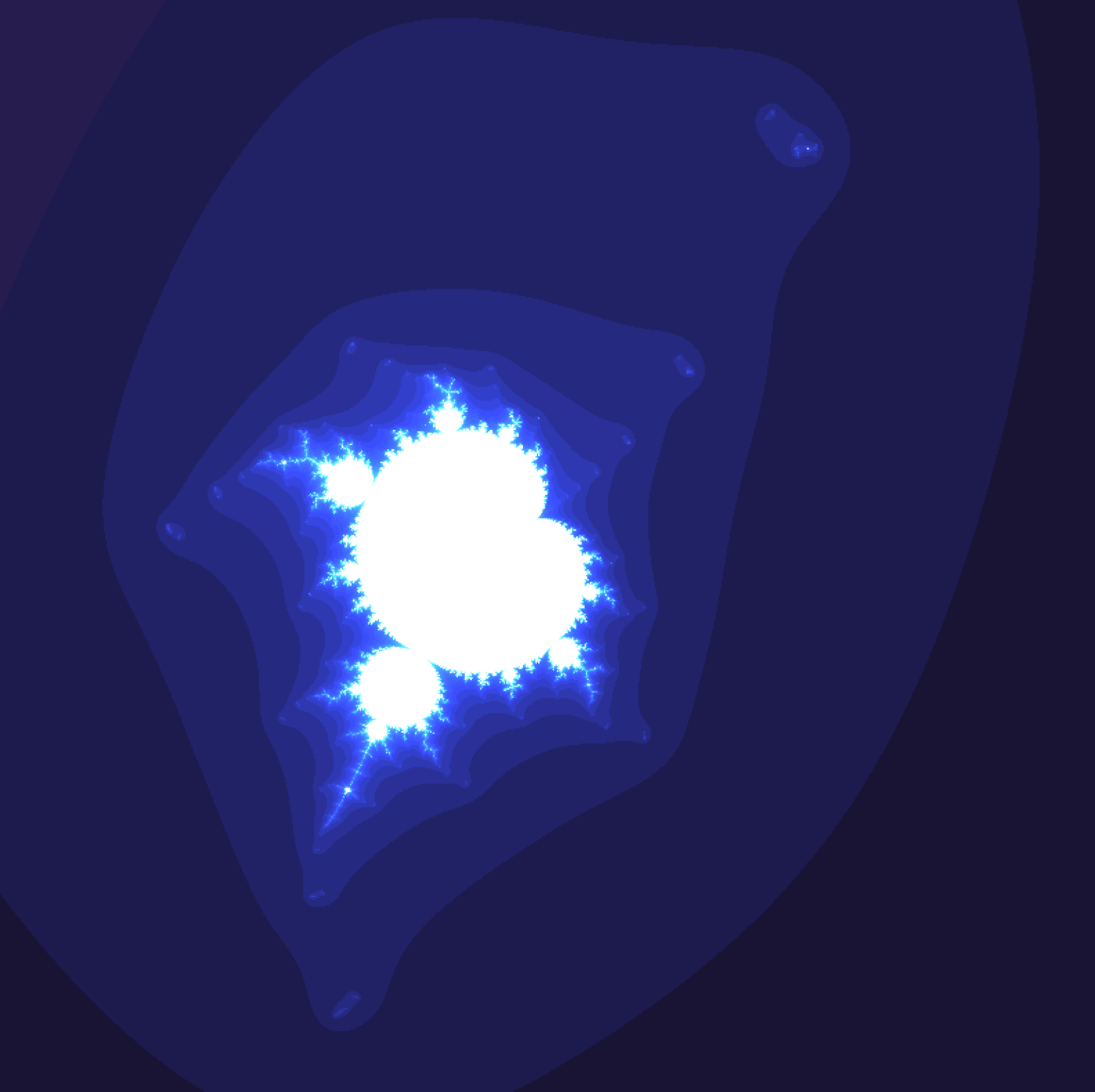}
\caption{Part of the degree 3 Shift locus (in blue) in a coordinate slice $f(z)=z^3+pz+1$.}
\label{approx_slice}
\end{figure}

It is computationally difficult to transform from B\"ottcher coordinates to
polynomial coordinates. Fortunately, because the shift locus is more or less foliated by
B\"ottcher slices, we may obtain a qualitatively reasonable picture of a
B\"ottcher slice by instead giving (part of) a `coordinate slice' of $\SS_3$
consisting of polynomials $z^3+pz+q$ where the linear term $q$ is fixed, at
least in a region where such a coordinate slice lies close to a B\"ottcher slice. 
Figure~\ref{approx_slice} is (part of) a `coordinate slice' of $\SS_3$
parameterizing shift polynomials of the form $z^3+pz+1$. 
There is one `large' continent, which resembles a
lopsided Mandelbrot, surrounded by a few visible small islands; and there is a little
archipelago to the northeast; compare with Figure~\ref{elam_and_pinched}.

There is a refinement of the tautological elamination (called the
{\em completed} tautological elamination) which (conjecturally) parameterizes the cut
points on the components of the complement of $\SS_q$ in a B\"ottcher slice. When
$q=2$ this recovers Thurston's combinatorial model for the cut points in the
Mandelbrot set. For a definition see \cite{Calegari_sausage}, \S~8.7.

\subsection{Degree 3, a worked example}

Tautological Laminations are laminations of the unit circle, which we normalize as
$S^1=\R/\Z$ so that it has length $1$. Tautological Laminations depend on a degree
$q\ge 2$ and a continuous parameter $C$ (morally, a vector of $q-2$ arguments of 
B\"ottcher coordinates) which, for $q=3$, is encoded by a single 
angle $\theta\in \R/\Z$.

Although the laminations depend on the parameter, the result of pinching the circle
to any finite depth does not. Thus, for each degree $q$ and each depth $n$ we obtain a
{\em partition} of $1$ into a vector of lengths of the components of $S^1 \mod \Lambda_n$.
These lengths are all integer multiples of $q^{-n}$ (in fact, they are of the form
$2^m q^{-n}$ for various $m$). Our goal is to count the number of
components of length exactly $q^{-n}$ (the {\em short components}), and the main result
of this section (Theorem~\ref{thm:recursive_formula}) gives the generating
function for the number of short components in degree $n$, and shows that it is
related in a rather simple way to the function $\beta(t)$ from 
Proposition~\ref{proposition:closed_form}.

We shall first give an ad hoc (though precise) 
definition in the special case $q=3$
and work out a few examples by hand.
Multiplication by 3 gives a map from $S^1$ to itself. If $\lambda$ is the leaf 
$\lbrace p,q\rbrace$ for distinct points $p,q \in S^1$ then $3\lambda$ is the leaf
$\lbrace 3p,3q \rbrace$.
For any $\theta\in S^1$ let $L(\theta)$ denote the leaf $\lbrace\theta,\theta+1/3\rbrace$.
By abuse of notation we let $L^-(\theta)$ be the limit of leaves $L(\theta-\epsilon)$
as $\epsilon \to 0$ from above, and we say that a leaf $\lambda$ {\em links}
$L^-(\theta)$ if it links $L(\theta-\epsilon)$ for all sufficiently small positive
$\epsilon$. Likewise, we let $L^+(\theta)$ be the limit of leaves $L(\theta+\epsilon)$
as $\epsilon \to 0$ from above.

\begin{definition}[Ad hoc definition, degree 3]
There will be one depth $n$ leaf of the tautological lamination
$\Lambda_T$ for each $x\in [1/3,2/3)$ for which $3^nx=0$. We claim
there is a unique $y \in S^1$ such that
\begin{enumerate}
\item{$3^ny=1/3$; and}
\item{if $\lambda$ denotes the leaf $\lbrace x,y\rbrace$ then $3^m\lambda$ does not
link $L^-(0)$ or $L^+(x)$ for $m=0,1,2,\cdots,n-1$.}
\end{enumerate}
Then $\lbrace 3x,3y\rbrace$ is a depth $n$ leaf of $\Lambda_T$, and all depth $n$ leaves arise
this way.
\end{definition}

\begin{example}[Depth 1]
The only $x\in [1/3,2/3)$ with $3x=0$ is $x=1/3$. The only $y$ with $3y=1/3$ and for which
$\lbrace 1/3,y\rbrace$ does not link $L^-(0)$ or $L^+(1/3)$ is $y=7/9$. Thus 
$\lbrace x,y\rbrace = \lbrace 1/3,7/9\rbrace$ and the leaf $3\lbrace x,y\rbrace = \lbrace 0, 1/3\rbrace$
is the unique depth 1 leaf of $\Lambda_T$. Thus $S^1 \mod \Lambda_T$ to depth 1
has two components of length $1/3$ and $2/3$ respectively.
\end{example}

\begin{example}[Depth 2]
For $x\in [1/3,2/3)$ with $9x=0$ we must have one of $x=1/3, 4/9, 5/9$. For $x=1/3$ we may
check for $\lambda=\lbrace 1/3,19/27\rbrace$ that $\lambda$ and $3\lambda$ do
not link $L^-(0)$ or $L^+(1/3)$. Likewise for $\lambda = \lbrace 4/9,22/27\rbrace$ that
$\lambda$ and $3\lambda$ do not link $L^-(0)$ or $L^+(4/9)$, and for
$\lambda = \lbrace 5/9,25/27\rbrace$ that $\lambda$ and $3\lambda$ do not link
$L^-(0)$ or $L^+(5/9)$. Thus the unique depth 2 leaves of $\Lambda_T$ are
$\lbrace 0,1/9\rbrace$, $\lbrace 1/3,4/9\rbrace$ and $\lbrace 2/3,7/9\rbrace$.
Thus $S^1\mod \Lambda_T$ to depth 2 has 3 components of length $1/9$, 1 component of length
$2/9$, and 1 component of length $4/9$.
\end{example}

\begin{example}[Depth 3]
For $x\in [1/3,2/3)$ with $27x=0$ we must have one of $x=1/3, 10/27, 11/27, \cdots , 17/27$.
Let's do one example. For $x=11/27$ we want $y$ with $27y=1/3$ so
for $\lambda = \lbrace 11/27,y\rbrace$ that $\lambda$, $3\lambda$, $9\lambda$ do not
link $L^-(0)$ or $L^+(11/27)$. One might naively guess (based on the examples in depth 1 and depth 2
in which $y=x+3^{-1}+3^{-n-1}$ is always the correct choice) that
$y=61/81$ would work, but $9\lbrace 11/27,61/81\rbrace = \lbrace 2/3, 7/9\rbrace$ which links
$L(11/27)=\lbrace 11/27, 20/27\rbrace$. In fact $y=58/81$, and $3\lbrace 11/27,58/81\rbrace =
\lbrace 2/9,4/27\rbrace$ is a depth 3 leaf of $\Lambda_T$.

One may check that $S^1 \mod \Lambda_T$ to depth 3 has $7$ components of length $1/27$,
$6$ components of length $2/27$, and 1 component of length $8/27$.
\end{example}

The set of leaves to depth 3 is
\begin{enumerate}
\item{$\lbrace 0,1/3\rbrace$;}
\item{$\lbrace 0,1/9\rbrace$, $\lbrace 1/3,4/9\rbrace$, $\lbrace 2/3,7/9\rbrace$;}
\item{$\lbrace 0,1/27\rbrace$, $\lbrace 1/9,7/27\rbrace$, $\lbrace 2/9,4/27\rbrace$,
$\lbrace 1/3,10/27\rbrace$, $\lbrace 4/9, 13/27\rbrace$, $\lbrace 5/9,16/27\rbrace$,
$\lbrace 2/3,19/27\rbrace$, $\lbrace 7/9, 22/27\rbrace$, $\lbrace 8/9, 25/27\rbrace$.}
\end{enumerate}
See the left side of Figure~\ref{elam_and_pinched}.

Continuing out to greater depth, the number of components of $S^1 \mod \Lambda_T$ to depth $n$
of length $3^{-n}$ is $1,3,7,21,57,171,499$ and so on.

\subsection{Tautological Laminations}\label{taut_lam_subsection}

Let us now give a more precise definition. Fix a degree $q$ which is an integer $\ge 2$. 
Multiplication by $q$ defines a degree $q$ map from $S^1$ to itself; if $\lambda$ is a leaf whose points
do not differ by a multiple of $1/q$ then it makes sense to define the leaf $q\lambda$.
Let $C$ be a finite lamination $C$ consisting of $q-2$ leaves $C_1,\cdots,C_{q-2}$
such that for each $j$ the points of $C_j$ differ by $1/q$. Thus we can write
$C_j = \lbrace \theta_j, \theta_j + 1/q\rbrace$ for some $\theta_j \in S^1$.
For simplicity we assume that $C$ is {\em generic}, meaning that the $\theta_j$ are
irrational and irrationally related.

The quotient $S^1 \mod C$ is a union of $q-1$ circles, $q-2$ of them of length
$1/q$ and one of length $2/q$; we refer to this as the {\em big circle} and denote it 
$B'$. The preimage of $B'$ in $S^1$ is a finite collection of arcs of 
$S^1$ bounded by points in leaves of $C$; denote this $B$. Note that the projection
from $B$ to $B'$ is 1--1 away from points in leaves of $C$.

We shall now define the depth $n$ leaves of the tautological lamination 
$\Lambda_T(C)$. Let $x \in B$ be a point for which $q^nx=\theta_j$; this $x$ will
be a point of a depth $n$ leaf of {\em type $j$}. Let $z \in B$
be the unique point for which the projections $x',z' \in B'$ are antipodal (i.e.\/
they are distance $1/q$ apart). Define a finite lamination $C(x)$ to be the union of $C$
together with the leaf $\lbrace x,z\rbrace$.

Now, $S^1 \mod C(x)$ is a union of $q$ circles, all of length $1/q$. Furthermore
if we denote by $\pi:S^1 \to S^1 \mod C(x)$ the projection (which is well-defined away 
from the points of $C(x)$) the map $z \to q \pi^{-1}z$ extends from its domain of
definition over the missing points to a homeomorphism from each component of 
$S^1 \mod C(x)$ to $S^1$. By abuse of notation, we denote this map by
$q:S^1 \mod C(x) \to S^1$ and think of it as a $q$--$1$ map.

\begin{lemma}[Division by $q$]\label{lemma:divide_by_q}
If $\lambda:=\lbrace a,b\rbrace$ is a generic leaf unlinked with $C(x)$, and $a' \in S^1$
satisfies $qa'=a$ then there is a unique $b'$ with $qb'=b$ so that $\lambda':=\lbrace
a',b'\rbrace$ is unlinked with $C(x)$.
\end{lemma}
\begin{proof}
`Generic' is just to rule out boundary cases where e.g.\/ $a'$ or $b'$ 
is equal to a point in $C(x)$. In particular, if $a'$ maps to a component
$S_i$ of $S^1 \mod C(x)$ then we can pull back $\lambda$ under the map
$q:S_i \to S^1$ and then take its preimage in $S^1$ to obtain $\lambda'$.
\end{proof}

Given $x$, we consider the sequence of points $x_i:=q^ix$ for $0 \le i \le n$.
By definition, $x_n=\theta_j$. Define $y_n=\theta_j+1/q$ and 
$\lambda_n:=\lbrace x_n,y_n\rbrace$ so that $\lambda_n = C_j$, and then
inductively let $\lambda_i$ be obtained from $\lambda_{i+1}$ as in 
Lemma~\ref{lemma:divide_by_q} so that $\lambda_i:=\lbrace x_i,y_i\rbrace$ where
$qy_i=y_{i+1}$ and $\lambda_i$ is unlinked from $C(x)$. Finally we obtain the
leaf $\lambda_0$ which, because it depends on $x$, we should really denote
$\lambda_0(x)$.

\begin{definition}\label{definition:leaf}
With notation as above, the depth $n$ leaves of $\Lambda_T(C)$ are the leaves
$q\lambda_0(x)$ of $S^1$ as $x$ ranges over the points in $B$ with $q^nx =\theta_j$
and $j$ ranges over $1,\cdots,q-2$.
\end{definition}
Notice that in this definition every leaf is enumerated exactly {\em twice}; if $x$ and
$z$ in $B$ have antipodal image in $B'$ then the images of $\lambda_0(x)$ and
$\lambda_0(z)$ are antipodal in $B'$ so that $q\lambda_0(x)=q\lambda_0(z)$. So we
only need to find a subset $A \subset B$ projecting to half of $B'$ and add
leaves $q\lambda_0(x)$ for $x \in A$ with $q^nx =\theta_j$. Thus the number of
leaves of depth $n$ is equal to $q^{n-1}(q-2)$. In particular $\Lambda_T(C)$ is empty 
if $q=2$.

\begin{proposition}[Lamination]
The leaves of $\Lambda_T(C)$ are pairwise unlinked; thus $\Lambda_T(C)$ really is a 
lamination. Furthermore, if $\Lambda_{T,n}(C)$ denotes the leaves of $\Lambda_T(C)$ 
of depth at most $n$, the set of lengths of components of $S^1 \mod \Lambda_{T,n}(C)$ 
(counted with multiplicity) is independent of $C$ and depends only on $q$.
\end{proposition}

For a proof, see \cite{Calegari_sausage}, \S~7. Tautological Laminations for
$q=3,4,5,6$ (for a rather symmetric choice of $C$) are displayed in Figure~\ref{lams}.

\begin{figure}[htpb]
\centering
\includegraphics[scale=0.33]{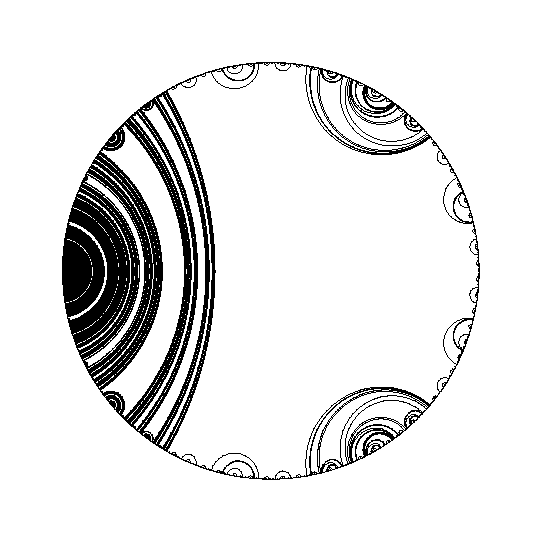}
\includegraphics[scale=0.33]{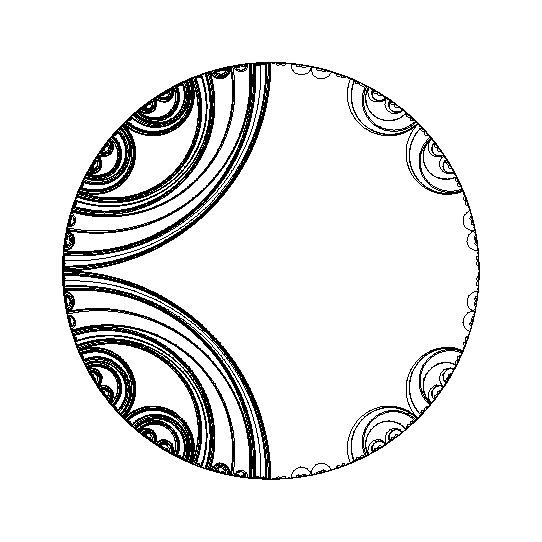}
\includegraphics[scale=0.33]{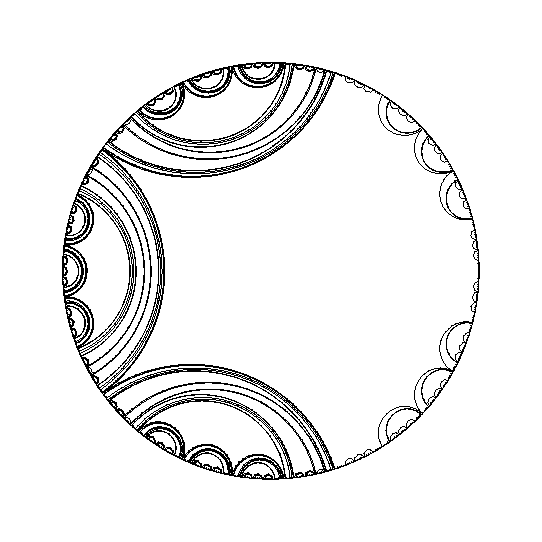}
\includegraphics[scale=0.33]{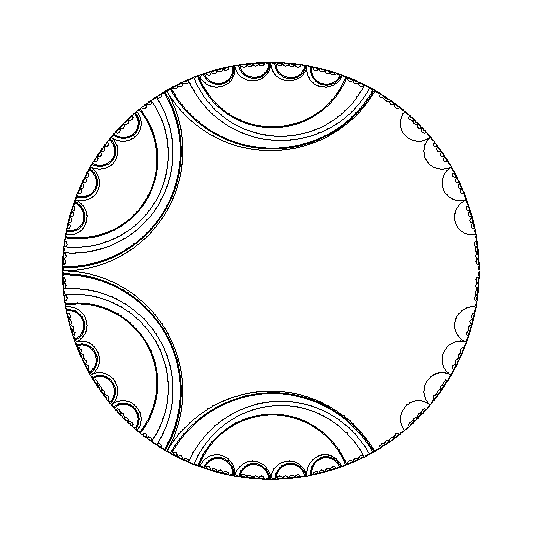}
\caption{Tautological Laminations for $q=3,4,5,6$.}
\label{lams}
\end{figure}

Since the set of lengths of $S^1 \mod \Lambda_{T,n}(C)$
(with multiplicity) is independent of $C$ we can fix a normalization
$\theta_j = (j-1)/q$ and suppress $C$ in our notation in the sequel. This set of
values is not generic; so we interpret the values of $\theta_j$ as limits as we
approach $(j-1)/q$ from below. So we should interpret $C_j$ as a `leaf' whose endpoints
span the interval $[(j-1)/q,j/q)$, for the purposes of determinining when leaves 
and their preimages are linked.

Then every depth $n$ leaf is of the form
$q\lambda$ where $\lambda = \lbrace x,y\rbrace$ and $q^nx,q^ny = (j-1)/q, j/q$
respectively. It follows that every depth $n$ leaf of $\Lambda_T$ consists of
a pair of points which are integer multiples of $q^{-n}$, and therefore every
component of $S^1 \mod \Lambda_{T,n}$ has
length which is an integer multiple of $q^{-n}$. What is not obvious, but is
nevertheless true, is that these integer multiples are all {\em powers of 2} (we shall
deduce this in the sequel). Write the length of a component as $\ell\cdot q^{-n}$
where $\ell$ is a power of 2, and define $N_q(n,m)$ to be the number of components of
$S^1 \mod \Lambda_{T,n}$ with $\ell=2^m$.

Let's spell out Definition~\ref{definition:leaf} in this normalization.
We can take $B$ and $A$ to be the half-open intervals
$$B=[(q-2)/q,1) \text{ and } A=[(q-2)/q,(q-1)/q)$$ 
The base $q$ expansion of $x \in A$ with $q^nx = (j-1)/q$ is
a word of length $n+1$ in the alphabet $\lbrace 0,1,\cdots,(q-1)\rbrace$ 
starting with the digit $(q-2)$ and ending with the digit $(j-1)$.
If we denote the digits of $x$ as $x_0 \cdots x_n$ then
$$x:= \cdot (q-2) x_1 x_2 \cdots x_{n-1} (j-1) \text{ and } z: = \cdot (q-1) x_1 x_2 \cdots x_{n-1} (j-1)$$
Likewise, we denote the digits of $y$ as $y_0 \cdots y_n$.
Then 
\begin{enumerate}
\item{$y_n = j$; and recursively,}
\item{if $x_i \ne q-2$ or $q-1$ then $y_i = x_i$; and}
\item{if $x_i = q-2$ or $q-1$ then $y_i$ is the unique one of $q-1$ or $q-2$
so that $\cdot x_i \cdots x_n$ and $\cdot y_i \cdots y_n$ do not link $x$ and $z$.}
\end{enumerate}

Although we are not able to give a simple formula for $N_q(n,m)$, it turns out there is
a relatively simple formula for $N_q(n,0)$ --- i.e.\/ the number of components of
$S^1 \mod \Lambda_{T,n}$ of length $q^{-n}$. These are the {\em short components}.

\subsection{Short Components}

One of the nice things about our normalization $C$ is that there is a simple
relationship between short components of $S^1 \mod \Lambda_{T,n}$ and certain
depth $n$ leaves of $\Lambda_{T,n}$, a relationship which is substantially more complicated
for longer components. Say a {\em short leaf} is a depth $n$ leaf
of $\Lambda_{T,n}$ whose points differ by exactly $q^{-n}$ 
(this is the least it can be). Then:

\begin{lemma}[Short Leaf]\label{lemma:short_leaf}
There is a bijection between short components and short leaves of any fixed depth.
\end{lemma}
\begin{proof}
Let $S$ be a short component at depth $n$, and consider the preimage $X$ in $S^1$. Then $X$ is
a union of finitely many disjoint arcs and isolated points 
bounded by leaves of depth $\le n$. The total length of $X$
is $q^{-n}$ by the definition of short component. But leaves of depth $k$ consist
of points which are integer multiples of $q^{-k}$ so the only possibility is that
$X$ consists of a single arc $Y$ of $S^1$ together with finitely many (possibly zero)
isolated points joined to the endpoints of $Y$ by a chain of leaves 
$\gamma_0,\cdots,\gamma_n$, each sharing one endpoint with the next.

We claim that in fact there are no isolated points, so that $X=Y$ is a single arc of
$S^1$ cut off by a single (necessarily) short leaf. 
To see this, let's enlarge the circle by a 
factor of $q^n$ so that depth $k$ leaves with $k<n$ 
consist of points which are divisible by $q$, and each depth $n$ leaf of type $j$
joins a point congruent to $(j-1)$ mod $q$ to a point congruent to $j$ mod $q$.
By the nature of their construction distinct depth $n$ leaves of type $j$ 
cannot share an endpoint, so a depth $n$ leaf of type $j$ must be followed by 
a depth $n$ leaf of type $j+1$, 
and only a type $1$ leaf of depth $n$ can follow a depth $<n$ leaf and only in the
positive (i.e.\/ anticlockwise) direction around $S^1$ (remember our understanding
of $\theta_j$ as the limit of a sequence approaching $(j-1)/q$ from below). 
It follows that if there is 
some intermediate point, the endpoints of $Y$ differ by at least $2$ mod $q$ so that
$S$ is not short after all. This proves the claim. \end{proof}

Note that this Lemma is {\em false} for generic $C$.

Let $\lambda'$ be a short leaf of $\Lambda_T$ of depth $n$ of the form $q\lambda_0(x)$
where $x\in A$ and $q^nx=\theta_j$. For this normalization, $z=x+1/q$ and 
$y=x+q^{-1-n}$ where $\lambda_0(x)=\lbrace x,y\rbrace$. 
The defining property of being a depth $n$ leaf means that 
$\lambda_k(x)=q^k\lambda_0(x)$ does not link $C(x)= C \cup \lbrace x,z\rbrace$ for any
$0\le k\le n$. Actually, for {\em any} integer $m \mod q^{n+1}$, setting $x= mq^{-1-n}$
and $y=(m+1)q^{-1-n}$, the leaf $\lambda_k(x)$ does not link any 
$C_i$ for $0\le i \le (q-2)$. So the short leaves are just the $x$ for which
$\lambda_k(x)$ does not link $\lbrace x,z\rbrace$ for $0\le k \le n$.

Remember that the base $q$ expansion of $x$ is a word of length $n+1$ in the
alphabet $\lbrace 0,1,\cdots,(q-1)\rbrace$ starting with the digit $(q-2)$ and
ending with the digit $(j-1)$. The base $q$ expansion of $y$ is the same as that of $x$
with the last digit replaced by $j$. Similarly, the base $q$ expansion of $z$ is the
same as that of $x$ with the first digit replaced by $(q-1)$. We deduce:

\begin{lemma}[Short is 1-unbordered]\label{lemma:short_unbordered}
A word in the alphabet $\lbrace 0,1,\cdots,(q-1)\rbrace$ of length $(n+1)$ 
starting with $(q-2)$ and ending with $(j-1)$ corresponds to a short leaf 
if and only if it is 1-unbordered.
\end{lemma}
\begin{proof}
The leaf $\lambda_k(x)=\lbrace q^kx,q^ky\rbrace$, and the base $q$ expansions of
$q^kx$ and $q^ky$ are obtained from the base $q$ expansions of $x$ and $y$ by
the $k$-fold left shift. This leaf links $\lbrace x,z\rbrace$ if and only if the 
length $k$-suffix of $x$ is either equal to a prefix of $x$, or to a prefix of $z$. 
But this is the definition of a 1-unbordered word. 
\end{proof}

Since $(j-1)$ is allowed to vary from $0$ to $(q-3)$, and since a word that
starts with $(q-2)$ and ends with $(q-2)$ or $(q-1)$ is already 1-bordered, 
it follows that $N_q(n,0)$ is equal to the number of 1-unbordered words of length
$(n+1)$ starting with $(q-2)$, which is just $q^{-1}$ times the number of 1-unbordered
words of length $(n+1)$. In other words: 

\begin{theorem}[Recursive Formula]\label{thm:recursive_formula}
$N_q(n,0)$ satisfies the recursion 
$N_q(0,0) = 1$, $N_q(1,0)=(q-2)$ and
$$N_q(2n,0) = qN_q(2n-1,0) \text{ and } N_q(2n+1,0) = qN_q(2n,0)-2N_q(n,0)$$
and has the generating function $(\beta(t)-1)/qt$ where
a closed form for $\beta(t)$ is given in Proposition~\ref{proposition:closed_form}.
\end{theorem}

\section{Tree Polynomials}\label{section:tree_polynomials}

We now discuss a rather different class of objects that turn out to be 
naturally in bijection with the components of $S^1 \mod \Lambda_{T,n}$. 
These objects are called {\em tree polynomials}.

We give our definition in terms of rooted trees (with some auxiliary planar structure)
and adopt the standard terminology of parents, children, siblings etc. Thus for every
(non-root) vertex there is a unique embedded path from that vertex to the root, and the {\em parent}
of $v$ is the unique vertex $w$ on that path connected to $v$ by an edge, and conversely
$v$ is the {\em child} of $w$; vertices are {\em siblings} if they share a common parent,
and so on.

\begin{definition}
A {\em tree polynomial} is a finite rooted tree $T$ together with the following
data:
\begin{enumerate}
\item{{\bf depth:} all leaves have a common depth $n$; we call this the {\em depth of $T$};}
\item{{\bf critical:} all vertices are {\em critical} or {\em ordinary};
\begin{enumerate}
\item{the root is critical;}
\item{every non-leaf critical vertex has exactly one critical child;}
\item{every ordinary vertex has no critical children;}
\end{enumerate}
}
\item{{\bf order:} the children of every vertex are {\em ordered}, and the
critical child of the root is first among its siblings;}
\item{{\bf self-map:} there is a simplicial self-map $f:T \to T$ such that
\begin{enumerate}
\item{$f(\text{root})=\text{root}$;}
\item{$f(v)=\text{root}$ for all children $v$ of the root; and}
\item{for all $v$ with non-root parent $w$, the image $f(v)$ is a
child of $f(w)$;}
\item{if $v$ is ordinary and not a leaf, then $f$ maps the children of $v$
bijectively and in an order-preserving way to the children of $f(v)$;}
\item{if $v$ is critical and not a leaf or the root, then $f$ maps the children of $v$
in an order non-decreasing way to the children of $f(v$); this map is onto and 
2--1 except for the critical child of $v$ which is the unique child 
mapping to its image;}
\end{enumerate}
}
\item{{\bf length:} there is a {\em length} function $\ell$ from the vertices to $\N$;
\begin{enumerate}
\item{$\ell(\text{root})=1$;}
\item{if $v$ is ordinary, $\ell(v)=\ell(f(v))$;}
\item{if $v$ is critical, $\ell(v)=2\ell(f(v))$.}
\end{enumerate}
}
\end{enumerate}
\end{definition}

Another way of talking about the order structure on the children of each vertex is
to say that $T$ is a {\em planar} tree, and the map $f$ is compatible with the
planar structure.

\subsection{Basic Properties}

\begin{definition}[Degree]
Let $T$ be a tree polynomial. The root has one critical child with $\ell=2$
and some non-negative number of ordinary children with $\ell=1$. All children
map to the root under $f$. Thus tree polynomials of depth $1$ are classified by
the number of children. The {\em degree} of a tree polynomial, denoted $q(T)$, 
is equal to the number of children of the root, plus one.
\end{definition}

\begin{example}[Degree 2]
There is a unique tree polynomial of degree $2$ of any positive depth, since
every vertex is critical and all but the leaf have a unique child.
\end{example}

\begin{definition}[Postcritical length]
Let $T$ be a tree polynomial and let $c$ be the unique critical
leaf. The {\em postcritical length} of $T$, denoted $\ell(T)$, 
is equal to $\ell(f(c))=\ell(c)/2$.
\end{definition}

By induction, $\ell(T)$ is always a power of $2$.

The next Proposition explains why we have introduced tree polynomials:
\begin{proposition}[Bijection]\label{prop:bijection}
There is a natural bijection between the set of degree $q$ tree polynomials $T$
of depth $(n+1)$ with $\ell(T)=\ell$ and the set of components of $S^1 \mod \Lambda_{n,T}$
of length $\ell\cdot q^{-n}$ where $\Lambda_{n,T}$ are the leaves of depth $\le n$ in
the Tautological Lamination from \S~\ref{section:lamination}.
\end{proposition}
\begin{proof}
This is a corollary of Thms~9.20 and 9.21 in \cite{Calegari_sausage}. The tree
polynomials are combinatorial abstractions of the sausage polynomials defined in 
\cite{Calegari_sausage} Def.~9.4. A sausage polynomial is a certain kind of infinite nodal
genus 0 Riemann surface $\Sigma$ together with a holomorphic 
self-map of degree $q$ satisfying a
number of properties. A tree polynomial records only the underlying combinatorics
of $\Sigma$, which is enough to recover $\ell$.
\end{proof}

It follows that for each $n$ and each $m$, the number of degree $q$ tree polynomials
$T$ of depth $(n+1)$ with $\ell(T) = 2^m$ is $N_q(n,m)$.

\begin{lemma}[Extension]\label{lemma:extension}
Let $T$ be a tree polynomial of depth $n$ and let $c$ be the unique critical 
leaf. Then tree polynomials $T'$ of depth $(n+1)$  
that extend $T$ are in bijection with the children of $f(c)$.
\end{lemma}
\begin{proof}
To extend $T$ to $T'$ we just add children to each of the leaves of $T$. For
each ordinary leaf $v$ we add a copy of the children of $f(v)$. For the
unique critical leaf $c$ we must choose a child $e$ of $f(c)$ and then add
as children of $c$ one copy of $e$, and two copies of every other child of $f(c)$.
The copy of $e$ becomes the unique critical child of $c$ in $T'$. The 
functions $f$ and $\ell$ extend to these new leaves uniquely.
\end{proof}

The next two lemmas give direct proofs in the language of tree polynomials of the
identities $\sum_\ell N_q(n,m)\cdot 2^m = q^n$ and $\sum_\ell N_q(n,m) = 1+(q-2)(q^n-1)/(q-1)$.
Both identities follow immediately from Proposition~\ref{prop:bijection} since
the first just says that the sum of the lengths of the components of
$S^1 \mod \Lambda_{T,n}$ is equal to 1, and the second just says that $\Lambda_T$ has
$(q-2)q^{n-1}$ leaves of depth $n$, both of which follow immediately from the definitions.

\begin{lemma}[Multiplication by $q$]\label{lemma:multiplication}
Let $T$ have degree $q$.
For each non leaf vertex $v$ with children $w_i$ we have 
$q\cdot\ell(v)= \sum_i \ell(w_i)$.
Consequently $\sum_\ell N_q(n,m)\cdot 2^m = q^n$.
\end{lemma}
\begin{proof}
There is a unique tree polynomial of depth $1$ and degree $d$. The
root has $\ell(\text{root})=1$, and it has $q-1$ children with $\ell=2,1,\cdots 1$.
Thus $q\cdot \ell(v)=\sum_i \ell(w_i)$ is true for the root vertex, 
and by induction on depth, it is true for each ordinary or critical non leaf vertex.

For any $T$, the extensions $T'$ of $T$ are in bijection with the 
children of the postcritical vertex, and the formula we just proved 
shows $\sum \ell(T') = q \ell(T)$.
\end{proof}

\begin{lemma}[Number of children]\label{lemma:number_of_children}
Let $T$ have degree $d$.
For each non leaf vertex $v$ the number of children of $v$ is $(q-2)\ell(v)+1$.
Consequently $$\sum_m N_q(n,m) = 1+(q-2)(q^n-1)/(q-1)$$
\end{lemma}
\begin{proof}
First we prove the formula relating $\ell(v)$ to the number of children of $v$.
The formula is true for the root vertex. If $v$ is ordinary then $\ell(v)=\ell(f(v))$
and $v$ has the same number of children as $f(v)$, so if the formula is true for $f(v)$
it is true for $v$. If $v$ is critical then $\ell(v) =2\ell(f(v))$ and if
$f(v)$ has $(q-2)\ell(f(v))+1$ children then $v$ has $2(q-2)\ell(f(v))+1$
children. So the formula is true by induction.

Since there are $N_q(n-1,m)$ depth $n$ degree $q$ tree polynomials of
length $\ell=2^m$, and since by Lemma~\ref{lemma:number_of_children} 
each has $(q-2)\ell(T)+1$ children, we obtain a recursion
$$\sum_m N_q(n,m) = \sum_m N_q(n-1,m)((q-2)2^m+1) = (q-2)q^{n-1}+ \sum_m N_q(n-1,m)$$
Since $\sum_m N_q(0,m)=1$ the lemma follows.
\end{proof}

\begin{lemma}[Length subdoubles]\label{lemma:subdoubles}
Every child $w$ of $v$ has $\ell(w)\le 2\ell(v)$ with equality iff every sibling
of $w$ has $\ell=1$.
\end{lemma}
\begin{proof}
By Lemma~\ref{lemma:multiplication} and Lemma~\ref{lemma:number_of_children}
$w$ has $(q-2)\ell(v)+1$ children, whose lengths sum to $q\ell(v)$.
\end{proof}

\begin{remark}
It is worth pointing out a close relationship between tree polynomials (as defined
above) and the {\em polynomial-like tree maps} of DeMarco--McMullen \cite{DeMarco_McMullen}.
The main difference seems to be that the latter objects forget the planar structure
(i.e.\/ the data of the ordering on each set of siblings). One should also mention that
there is a close relationship between the {\em dynamical elaminations} 
(see \cite{Calegari_sausage}, \S~4.4) and the {\em pictographs} of DeMarco--Pigrim
which are in turn closely related to the pattern and tableau of Branner--Hubbard 
\cite{Branner_Hubbard_2}.
The tautological elaminations we discuss in this article are (roughly speaking) 
related to dynamical elaminations as the Shift locus is related to individual shift polynomials.
\end{remark}

\section{$F$-sequences}

\begin{definition}[Critical Vein]
Let $T$ be a tree polynomial of depth $n$. The {\em critical vein} is the
segment of $T$ containing all the critical vertices. We denote it by $\gamma$
and label the critical points on $\gamma$ as $c_i$ where $c_0$ is the root and
$c_n$ is the critical leaf.
\end{definition}

For each vertex $w$ of $T$, define $F(w)$ to be equal to $f^k(w)$ for the
least positive $k$ so that $f^k(w)$ is critical. Thus we can think of $F$
as a map from the critical vein to itself, and by abuse of notation, for
integers $i,j$ we write $F(i)=j$ if $F(c_i)=c_j$ so that we can and do think
of $F$ as a function from $\lbrace 0,\cdots, n\rbrace$ to itself. We also write 
$\ell(i)$ for $\ell(c_i)$.

\begin{lemma}[Properties of $F$]
$F(0)=0$ and for every positive $i$, $F(i)<i$ and $F(i+1)\le F(i)+1$.
Furthermore, $\ell(i)=2\ell(F(i))$ for $i>0$ so that $\ell(i)=2^k$ 
where $k$ is the least integer so that $F^k(i)=0$
\end{lemma}
\begin{proof}
Since $f(w)$ has smaller depth than $w$ unless $w$ is the root, $F(i)<i$.
Furthermore, $F(c_{i+1})$ is equal to $F(w)$ for some child $w$ of 
$F(c_i)$, so $F(i+1)\le F(i)+1$. 

Finally, $\ell(w)=\ell(f(w))$ when $w$ is ordinary or the root, 
and $\ell(w)=2\ell(f(w))$ when $w$ is critical and not the root.
\end{proof}

Let $\gamma^+$ denote the union of $\gamma$ together with the siblings
of every critical vertex. We may think of it as a digraph (i.e.\/ a directed
graph), where every edge points away from the root. 
Define $\Gamma$ to be the quotient
of $\gamma^+$ obtained by identifying every sibling $w$ of a critical 
vertex with its image $F(w)$. Note that $\Gamma$ is a digraph. 

The next proposition gives a characterization of the functions $F$ that 
can arise from tree polynomials.

\begin{proposition}[$F$-sequence]\label{prop:F_sequence}
A function $F$ from $\lbrace 0,1,\cdots,n \rbrace$
to itself arises from some tree polynomial $T$ of depth $n$ if and only if
it satisfies the following properties:
\begin{enumerate}
\item{$F(0)=0$ and $F(1)=0$;}
\item{each $j$ has a finite set of {\em options} which are the admissible
values of $F(i+1)$ when $F(i)=j$;
\begin{enumerate}
\item{the options of $0$ are $0$ and $1$;}
\item{if $F(i)=0$ then the options of $i$ are $i+1$ and whichever option of 
$0$ is not equal to $F(i+1)$;}
\item{if $F(i)\ne 0$ and $F(i+1)$ is not equal to $F(i)+1$ then the options
of $i$ are $i+1$ together with the options of $F(i)$;}
\item{if $F(i)\ne 0$ and $F(i+1)= F(i)+1$ then the options of $i$ are 
$i+1$ together with all the options of $F(i)$ except $F(i)+1$.}
\end{enumerate}}
\end{enumerate}
A function $F$ is an {\em $F$-sequence} if it satisfies these properties.
\end{proposition}
\begin{proof}
Imagine growing a tree $T$ by iterated extensions from a
tree of depth $1$. The extensions at each stage are the children of
the critical image $f(v)$, which in turn may be identified with the 
children of $F(v)$.

When we grow $T$ of depth $i$ to $T'$ we grow $\Gamma$ to $\Gamma'$ by adding a new
edge from $i$ to a new vertex $i+1$, and adding two new edges 
from $i$ to $j$ for each edge from $F(i)$ to $j$ {\em except} 
for the edge from $F(i)$ to $F(i+1)$ (outgoing edges at $F(i)$ may be identified
with the {\em options} of $F(i)$ as above). The root $0$ is joined by a single
edge both to itself and to $1$, but every other vertex $i$ is joined by a
single edge to $i+1$ and by an {\em even} number of edges to each of its
options (this can be seen by induction). The proposition follows.
\end{proof}

\begin{remark}
The referee has pointed out that the $F$-function carries essentially the same information 
as the Yoccoz' $\tau$-function derived from the tableau of Branner--Hubbard 
(see e.g.\/ \cite{Branner_Hubbard_2}, \S~4.2 and \cite{DeMarco_McMullen}, \S~11).
Proposition~\ref{prop:F_sequence} is essentially equivalent to Prop.~2.1 from \cite{DeMarco_Schiff}.
\end{remark}

The map from tree polynomials to $F$-sequences is many to one, since
for every vertex $i$ of $\Gamma$ except the root, if there is an edge
from $i$ to $j \ne i+1$ then there are at least two such edges. 
Nevertheless, if $F$ is an $F$-sequence corresponding to $T$ of depth $n$,
the extended sequence defined by $F(n+1)=F(n)+1$ corresponds to the unique
extension of $T$ for which $F(c_{n+1})$ is the critical child of $F(c_n)$.

\begin{example}
$\ell(T)=1$ if and only if $F(n)=0$ where $n$ is the depth of the tree.
The number of depth $n$ tree polynomials with this property is $N_q(n-1,0)$ by definition.
\end{example}

\begin{example}[Maximal Type]
There is a unique depth $n$ tree polynomial $T$ of any degree with 
$2\ell(T)=\ell(n)=2^n$ namely the tree polynomial for which $F(i)=i-1$ for
all positive $i$. Thus $N_q(n,n)=1$. We call these trees of {\em maximal type}.

The components of $S^1 \mod \Lambda_{T,n}$ corresponding to the 
trees of maximal type are clearly evident in Figure~\ref{lams} (they correspond
to the large components of `white space').
\end{example}

\begin{example}[Degree 3 Maximal Component]\label{example:maximal_component}
Consider the tree polynomial sequence of maximal type of degree 3.
Let's work, as in \S~\ref{taut_lam_subsection} in the normalization
$C=\lbrace 0,1/3\rbrace$. For each $n$ the result of pinching
$S^1\mod \Lambda_{T,n}(C)$ has a unique component of length $2^n/3^n$, and
this sequence of components corresponds precisely to the (degree 3) tree
polynomial sequence of maximal type; we call this component, for each $n$,
the {\em maximal component}.

For each $n$ we can let $K_n(C)$ be the preimage of the maximal component in $S^1$.
As a subset of $S^1$ this depends on $C$, but for the specific normalization
$C=\lbrace 0,1/3\rbrace$ it does not, and we abbreviate $K_n(C)=K_n$.
It turns out that there is a very explicit description of $K_n$: it consists
of numbers in $[0,1)$ whose base 3 expansion is of two types:
\begin{enumerate}
\item{the first $n$ digits contain no 0; or}
\item{numbers of the form $\cdot x \dot{0}$ where $x$ is a string of $<n$ digits
containing no 0.}
\end{enumerate}
In other words:
$$K_1 = 0\cup [1/3,1), \quad K_2 = 0\cup 1/3\cup [4/9,2/3]\cup [7/9,1), \quad \text{etc.}$$
If we denote $K:=\cap_n K_n$, then $K$ is the set of numbers in $[0,1]$ 
whose base three expansion either contains no 0s, or is of the form $\cdot x\dot{0}$
for some finite string $x$ containing no 0s. See Figure~\ref{bigbrot}.
Compare with the left side of Figure~\ref{elam_and_pinched}. This component
corresponds to the `lopsided Mandelbrot' in Figure~\ref{approx_slice}.
\begin{figure}[htpb]
\centering
\includegraphics[scale=0.7]{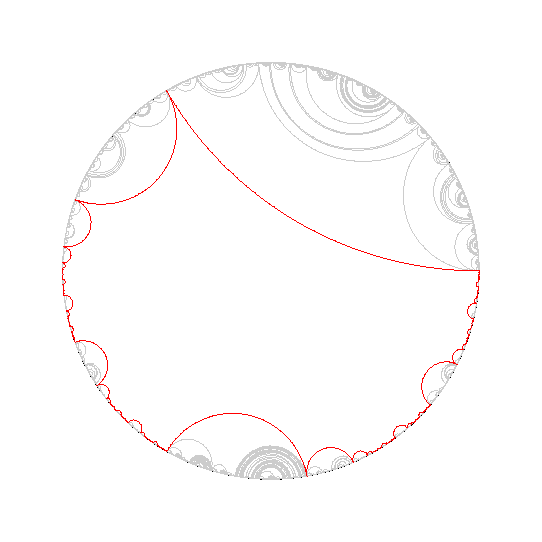}
\caption{The maximal component $K$ of $S^1 \mod \Lambda_T$}
\label{bigbrot}
\end{figure}

The `boundary' of the region $K$ consists of a union of $2^{n-1}$ short components
of each positive depth $n$, as follows from Lemma~\ref{lemma:short_unbordered},
and no other components (because otherwise the length of the maximal component
would be strictly less than $2^n/3^{-n}$ for some $n$).
\end{example}

For each positive $k$ let $i_k$ denote the least index (if any)
for which $\ell(i_k)=2^k$. Note that $i_0=0$.

\begin{lemma}[Increments Grow]\label{lemma:increments_grow}
$F(i_k)=i_{k-1}$ for $k>0$. Consequently $|i_{k+1}-i_k| \ge |i_k - i_{k-1}|$.
\end{lemma}
\begin{proof}
By definition $F(i_k)$ is some value of $j$ with $\ell(j)=2^{k-1}$. But
if $j>i_{k-1}$ there was some $i' < i_k$ with $F(i')=i_{k-1}$, contrary to the
definition of $i_k$.

The inequality follows from $F(i+1) \le F(i)+1$.
\end{proof}

\begin{definition}[$S$ and $B$]
Let $F$ be an $F$-sequence. Let $S$ be the set of indices $i$ such that 
$F(i+1)=F(i)+1$ and let $B$ be the rest. Note that $0$ is in $B$.
\end{definition}

The {\em prior options} of $i$ are the options other than $i+1$. We denote
these by $P(i)$. Thus, if $i \in S$ then $P(i)=P(F(i))$.

\begin{lemma}[Backslide]\label{lemma:backslide}
Let $F$ be an $F$-sequence and let $i\in B$.
Then $F(i+1)<F(i)$ and $F(i+1) \in P(b)$ where $b=F^k(i)$ for some $k$
and $b \in B$.
\end{lemma}
\begin{proof}
Since $i \in B$ we must have $F(i+1) \in P(F(i))$ so that necessarily 
$F(i+1)<F(i)$. Furthermore, if $F(i) \in S$, then $P(F(i))=P(F^2(i))$ and
so on by induction until the first $k$ so that $F^k(i) \in B$.
\end{proof}

Using $F$-sequences we may deduce the following `gap' theorem, that was
observed experimentally.

\begin{theorem}[Gap]\label{thm:gap_theorem}
$N_q(n,m)=0$ for $\lfloor n/2 \rfloor < m < n$.
\end{theorem}
\begin{proof}
Let $T$ be a tree polynomial of depth $n+1$. If $m<n$ then $T$ is not of
maximal type, so there is some first positive index $k \in B$. Note that
$i_k = k$ and in fact $i_j=j$ for all $j\le k$. 
Since $k \in B$, by Lemma~\ref{lemma:backslide}, $F(k+1)\in P(b)$
where $b\in B$ is $<k$. But then $b=0$ so $F(k+1)=0$.
It follows that $i_{k+1}\ge 2k+1$ and, successively, $i_{k+j} \ge k+j(k+1)$.
From this the desired inequality follows.
\end{proof}

\section{Tautological Tree}

Degree $q$ tree polynomials of various depth 
can themselves be identified with the vertices of a (rooted, planar)
{\em tautological tree} $\T_q$, whose vertices at depth $n$ are the 
tree polynomials of degree $q$ and depth $n$, and for each vertex $T$ of $\T_q$,
the children of $T$ are the extensions of $T$.

Note that for each vertex $T$ of $\T_q$ we can recover $\ell(T)$ from the
number of children of $T$ in $\T_q$, since this number is $(q-2)\ell(T)+1$.
So all the data of $N_q(n,m)$ can be read off from the abstract underlying
tree of $\T_q$ (in fact, even the root can be recovered from the fact that it
is the unique vertex of valence $q-1$).

The tree $\T_3$ up to depth $4$ (with vertices labeled by $\ell$ value, from
which one could easily extend it another row as an unlabeled tree) 
is depicted in Figure~\ref{TT_3}:

\begin{figure}[htpb]
\centering
\includegraphics[scale=0.8]{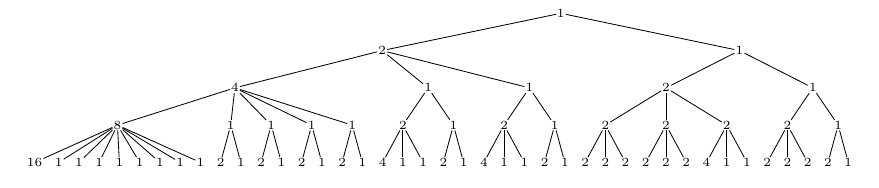}
\caption{$\T_3$ up to depth $4$.}
\label{TT_3}
\end{figure}

Every vertex labeled $1$ has two children labeled $2$ and $1$. Every vertex
labeled $2$ has three children, but these might be labeled $4,1,1$ or $2,2,2$.
Components of the complement of the shift locus $\SS_q$ in a 1-dimensional slice 
are in bijection with the ends of $\T_q$. Each such end gives rise to a 
sequence $\ell(n)$ of $\ell$-values, and when $\sum 1/\ell(n)$ diverges, 
the corresponding component consists of a single point. Such ends are called
{\em small}; those with $\sum 1/\ell(n) < \infty$ are {\em big}.
Big ends are {\em dense} in the space of ends of $\T_q$:

\begin{proposition}[Big ends dense]\label{prop:big_ends_dense}
Big ends are dense. In other words, every finite rooted path in $\T_q$ can
be extended to an infinite path converging to a big end.
\end{proposition} 
\begin{proof}
Let $T$ be a tree polynomial of some finite depth $n$ and let $F$ be the
associated $F$-sequence. Suppose $F(n)=i$. There is a unique infinite 
sequence of extensions of $T$ defined recursively by $F(m+1)=F(m)+1$ for all 
$m\ge n$. Then $F(m)=m+i-n$ for all sufficiently large $m$, so that 
$\ell(m)=2\ell(m+i-n)$ and $\sum 1/\ell(m) < \infty$.
\end{proof}

Let's call an end {\em type $S$} if the associated $F$-sequence
satisfies $F(m+1)=F(m)+1$ for all sufficiently large $m$ (i.e.\/ if it is
of the sort constructed in Proposition~\ref{prop:big_ends_dense}). 
For example, the sequence of maximal type is of type $S$.

\begin{example}[Littlebrot]
The right side of Figure~\ref{littlebrot} 
depicts the second biggest complementary component 
in a B\"ottcher's slice (this is a speck in the northeast corner in Figure~\ref{approx_slice}).
It corresponds to an end of type $S$ with $\ell(n)=2^{\lfloor n/2\rfloor}$.

\begin{figure}[htpb]
\centering
\includegraphics[scale=0.65]{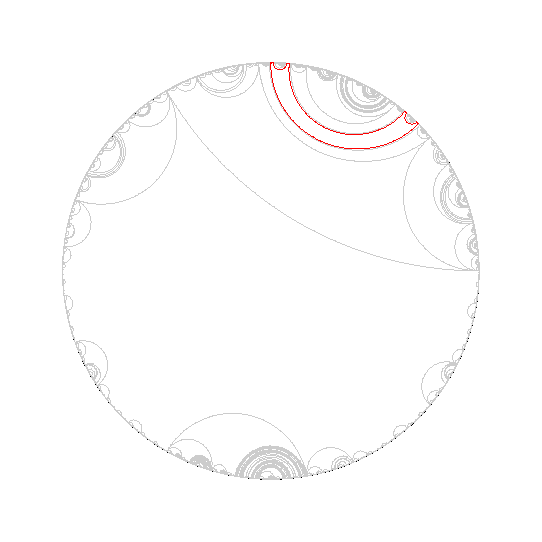} \quad
\includegraphics[scale=0.2]{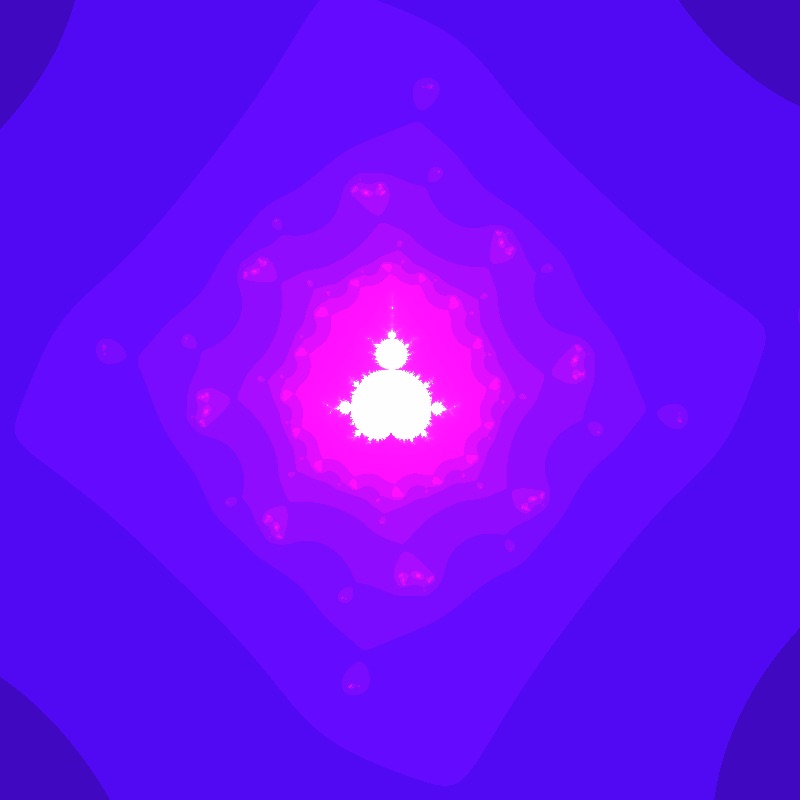}
\caption{A slice through $z^3+pz+2$ of width 0.0003 centered at $1.72572+3.09778i$
and the corresponding component of the tautological lamination}
\label{littlebrot}
\end{figure}
In the normalization $C=\lbrace 0,1/3\rbrace$ 
the base 3 decimal expansions of the points in the subset of $S^1$ associated to this component 
is a regular language in the alphabet $\lbrace 0,1,2\rbrace$ (whose precise description is
somewhat complicated and not very enlightening). Compare with 
Example~\ref{example:maximal_component}.
\end{example}

Branner--Hubbard \cite{Branner_Hubbard_2}, Thm.~9.1 implies (in degree 3, but the same result
should hold in every degree) that every big end is of type $S$
and a component of the complement of $\SS_q$ in a slice has positive diameter 
if and only if it corresponds to a big end of $\T_q$.

\begin{conjecture}
In the normalization $C=\lbrace 0,1/3\rbrace$, every big end corresponds to a subset of 
$S^1$ whose base 3 decimal expansion is a regular language in the alphabet $\lbrace 0,1,2\rbrace$.
\end{conjecture}

\section{Tables of Values}\label{section:tables}

Values of $N_3(n,m)$ for $0 \le n,m \le 12$ are contained in Table~\ref{n3table}.
Values of $N_q(n,m)$ for $0 \le n,m \le 11$ and $q=4,5$ are in Tables~\ref{n4table}
and \ref{n5table}. These tables were computed with the aid of the program 
{\tt taut}, written by Alden Walker.

\begin{table}[ht]
{\small
\centering
\begin{tabular}{c|c c c c c c c c c c c c c} 
 $n\backslash m$ & 0 & 1 & 2 & 3 & 4 & 5 & 6 & 7 & 8 & 9 & 10 & 11 & 12 \\
 \hline
 0 &  1\\ 
 1 &  1 & 1\\
 2 &  3 & 1 & 1\\
 3 &  7 & 6 & 0 & 1\\
 4 &  21 & 16 & 3 & 0 & 1\\
 5 &  57 & 51 & 13 & 0 & 0 & 1\\
 6 &  171 & 149 & 39 & 5 & 0 & 0 & 1\\
 7 &  499 & 454 & 117 & 23 & 0 & 0 & 0 & 1\\
 8 &  1497 & 1348 & 360 & 66 & 9 & 0 & 0 & 0 & 1\\
 9 &  4449 & 4083 & 1061 & 207 & 41 & 0 & 0 & 0 & 0 & 1\\
 10 & 13347 & 12191 & 3252 & 591 & 126 & 17 & 0 & 0 & 0 & 0 & 1\\
 11 & 39927 & 36658 & 9738 & 1799 & 370 & 81 & 0 & 0 & 0 & 0 & 0 & 1\\
 12 & 119781 & 109898 & 29292 & 5351 & 1125 & 240 & 33 & 0 & 0 & 0 & 0 & 0 & 1\\
 \hline
\end{tabular}
}
\vskip 12pt
\caption{Values of $N_3(n,m)$}\label{n3table}
\end{table}

\begin{table}[ht]
{\small
\centering
\begin{tabular}{c|c c c c c c c c c c c c c} 
 $n\backslash m$ & 0 & 1 & 2 & 3 & 4 & 5 & 6 & 7 & 8 & 9 & 10 & 11\\
 \hline
 0 & 1\\
 1 & 2 & 1\\
 2 & 8 & 2 & 1\\
 3 & 28 & 14 & 0 & 1\\ 
 4 & 112 & 52 & 6 & 0 & 1\\
 5 & 432 & 220 & 30 & 0 & 0 & 1\\ 
 6 & 1728 & 872 & 120 & 10 & 0 & 0 & 1\\
 7 & 6856 & 3540 & 472 & 54 & 0 & 0 & 0 & 1\\
 8 & 27424 & 14120 & 1924 & 204 & 18 & 0 & 0 & 0 & 1\\ 
 9 & 109472 & 56744 & 7620 & 828 & 98 & 0 & 0 & 0 & 0 & 1\\ 
 10 & 437888 & 226768 & 30752 & 3212 & 396 & 34 & 0 & 0 & 0 & 0 & 1\\ 
 11 & 1750688 & 908040 & 122852 & 12872 & 1556 & 194 & 0 & 0 & 0 & 0 & 0 & 1\\
 \hline
\end{tabular}
}
\vskip 12pt
\caption{Values of $N_4(n,m)$}\label{n4table}
\end{table}

\begin{table}[ht]
{\small
\centering
\begin{tabular}{c|c c c c c c c c c c c c c} 
 $n\backslash m$ & 0 & 1 & 2 & 3 & 4 & 5 & 6 & 7 & 8 & 9 & 10 & 11\\
 \hline
 0 & 1\\
 1 & 3 & 1\\
 2 & 15 & 3 & 1\\
 3 & 69 & 24 & 0 & 1\\
 4 & 345 & 114 & 9 & 0 & 1\\
 5 & 1695 & 597 & 51 & 0 & 0 & 1\\
 6 & 8475 & 2973 & 255 & 15 & 0 & 0 & 1\\
 7 & 42237 & 15018 & 1245 & 93 & 0 & 0 & 0 & 1\\
 8 & 211185 & 75012 & 6306 & 438 & 27 & 0 & 0 & 0 & 1\\
 9 & 1055235 & 375951 & 31287 & 2199 & 171 & 0 & 0 & 0 & 0 & 1\\
 10 & 5276175 & 1879269 & 157098 & 10767 & 858 & 51 & 0 & 0 & 0 & 0 & 1\\
 11 & 26377485 & 9400644 & 784596 & 53799 & 4230 & 339 & 0 & 0 & 0 & 0 & 0 & 1\\
 \hline
\end{tabular}
}
\vskip 12pt
\caption{Values of $N_5(n,m)$}\label{n5table}
\end{table}

\section{Acknowledgements}

I would like to thank Laurent Bartholdi, Frank Calegari, Eric Rains and Alden Walker;
furthermore I would like to especially thank the anonymous referee for extensive and detailed
feedback on several previous iterations of this article.

\end{document}